\documentclass{amsart}
\usepackage{subfigure,fancyhdr, pstricks}
\usepackage{amssymb,amsmath,amsthm,inputenc}
\usepackage{bbm}
\usepackage[dvips]{epsfig}
\usepackage[dvips]{graphicx}
\usepackage{pst-all}
\usepackage{multido}
\usepackage{oldgerm}
\usepackage[all]{xy}
\let\SAVEDRightarrow=\Rightarrow
\usepackage{marvosym}
\let\Rightarrow=\SAVEDRightarrow
\usepackage{pifont}
\usepackage{array}
\usepackage{graphicx}
\newgray{gray}{0.50}
\newrgbcolor{royalblue}{0.25 0.41 0.88}
\newrgbcolor{green}{0.1 0.8 0.2}
\newrgbcolor{lightgreen}{0 1 0}
\newrgbcolor{ygreen}{0.7 1 0.7}
\newrgbcolor{lblue}{0.8 0.83 1}
\theoremstyle{plain}

\newtheorem{satz}{Satz}[section]

\newtheorem{lemma}[satz]{Lemma}

\newtheorem{prop}[satz]{Proposition}
\newtheorem{thm}[satz]{Theorem}

\theoremstyle{definition}
\newtheorem{defi}[satz]{Definition}

\newtheorem{rmrk}[satz]{Remark}

\newtheorem{example}[satz]{Example}

\theoremstyle{remark}

\DeclareMathOperator{\Cone}{Cone}

\DeclareMathOperator{\adn}{adn}

\renewcommand{\phi}{\varphi}
\renewcommand{\theta}{\vartheta}
\renewcommand{\epsilon}{\varepsilon}
\renewcommand{\rho}{\varrho}
\newcommand{\IN}{\mathbbm{N}}
\newcommand{\IZ}{\mathbbm{Z}}

\newcommand{\IR}{\mathbbm{R}}

\newcommand{\bullseyefigure}{\begin{center}
\psset{labelsep=2pt}
\begin{pspicture}(-5,-5)(5,5)
\psdot(0,0)
\pscircle(0,0){4}
\pscircle(0,0){2}
\pscircle(0,0){1}
\pscircle(0,0){0.5}
\pscircle(0,0){0.25}
\pscircle(0,0){0.125}
\psline(-5,0)(5,0)
\psline(0,0.125)(0,0.25)
\psline(0,0.5)(0,1)
\psline(0,1)(0,2)
\psline(0,4)(0,5)
\psdot[dotstyle=triangle*, dotscale=0.5](-0.375,0)
\psdot[dotstyle=triangle*, dotscale=1](-0.75,0)
\psdot[dotstyle=triangle*, dotscale=2](-1.5,0)
\psdot[dotstyle=triangle*, dotscale=4](-3,0)
\psdot[dotstyle=square*, dotscale=0.5](0.375,0)
\psdot[dotstyle=square*, dotscale=1](0.75,0)
\psdot[dotstyle=square*, dotscale=2](1.5,0)
\psdot[dotstyle=square*, dotscale=4](3,0)
\uput[r](0,0.72){\scriptsize 1}
\uput[r](0,1.4){\scriptsize 1}
\uput[r](0,2.9){\scriptsize 0}
\uput[r](0,4.6){\scriptsize 1}
\end{pspicture}~\\
{\scriptsize Figure 1: A bullseye space.}
\end{center}}
\numberwithin{equation}{section}

\DeclareFontFamily{U}{schwell}{}
\DeclareFontShape{U}{schwell}{m}{n}{
   <8> <9> <10> <10.95> <12> <14.4> <17.28>  <20.74> <24.88> schwell}{}
\DeclareMathAlphabet{\schwell}{U}{schwell}{m}{n}
\newcommand\textschwell{\usefont{U}{schwell}{m}{n}}
\DeclareTextFontCommand{\schwell}{\textschwell}
\DeclareFontFamily{U}{suet}{}
\DeclareFontShape{U}{suet}{m}{n}{
   <8> <9> <10> <10.95> <12> <14.4> <17.28>  <20.74> <24.88> suet14}{}
\DeclareMathAlphabet{\suet}{U}{suet}{m}{n}
\newcommand\textsuet{\usefont{U}{suet}{m}{n}}
\DeclareTextFontCommand{\suet}{\textsuet}
\DeclareMathAlphabet{\dis}{T1}{cmss}{bx}{sl}
\input{cyracc.def}
\newfont{\cyrfnt}{wncyr10}

\newfont{\cybfnt}{wncyb10}

\newfont{\cyifnt}{wncyi10}

\newfont{\cyscfnt}{wncysc10}

\newfont{\cyssfnt}{wncyss10}

\author{Lars Scheele}
\author{Alessandro Sisto}

\address{Universit\"at M\"unster, Einsteinstr. 60, 48149 M\"unster, Germany}

\address{Mathematical Institute, 24-29 St Giles, Oxford OX1 3LB, United Kingdom}
 
\email{lars.scheele@uni-muenster.de, sisto@maths.ox.ac.uk}

\title[Iterated Asymptotic cones]{Iterated asymptotic cones}


\begin{document}

\maketitle
\begin{abstract}
 Iterated asymptotic cones have been used by Dru\c{t}u and Sapir to construct a group with uncountably many pairwise non-homeo-morphic asymptotic cones. In this paper we define a class of metric spaces which display a wide range of behaviors with respect to iterated asymptotic cones, and we use those to construct examples within the class of groups. Namely, we will show that there exists a group whose iterated cones are pairwise non-homeomorphic, or periodically homeomorphic.
\end{abstract}

\section*{Introduction}
Asymptotic cones are useful quasi-isometry invariants of metric spaces, especially groups, each of which encodes some aspects of the large scale geometry of a metric space. Gromov put forward the idea to construct them in \cite{G1}, and this idea has been refined by van den Dries and Wilkie in \cite{VDW}. The asymptotic cones of the metric space $X$ depend on the choice of an ultrafilter, a sequence of points of $X$ (called \emph{base-point}) and a divergent sequence of positive real numbers (called \emph{scaling factor}). The asymptotic cones of a group do not depend on the choice of the base-point and this is why there has been more interest in the dependence on the ultrafilter/scaling factor (those choices are interrelated). Roughly speaking, the asymptotic cones of $X$ do not depend on the scaling factor if $X$ ``looks the same'' at each large scale. Many classes of groups have asymptotic cones which do not depend (up to bilipschitz homeomorphism) on the scaling factor/ultrafilter, for example abelian groups, nilpotent groups \cite{P}, lattices in $SOL$ \cite{dC}, hyperbolic groups \cite{DP}, groups hyperbolic relative to subgroups whose asymptotic cones do not depend on the scaling factor/ultrafilter \cite{OsS}, \cite{Si2}. The first example of a (non finitely presented) group whose asymptotic cones do depend on the scaling factor has been found by Thomas and Velickovic in \cite{TV}, while the first finitely presented example is due to Ol'shanskii and Sapir \cite{OS}.
\par
A striking example of dependence on the ultrafilter has been found in \cite{KSTT}, where it is shown that there are groups (namely lattices in certain Lie groups) all whose asymptotic cones are homeomorphic if the Continuum Hypothesis $(CH)$ holds, while they have $2^{2^{\aleph_0}}$ pairwise non-homeomorphic asymptotic cones if $(CH)$ fails. That paper also contains the proof that under $(CH)$ a group can have at most $2^{\aleph_0}$ pairwise non-homeomorphic asymptotic cones.
\par
Dru\c{t}u and Sapir \cite{DS} provided an example of a group which achieves the maximum (under $(CH)$) cardinality $2^{\aleph_0}$ of pairwise non-homeomorphic asymptotic cones, and in order to construct such an example they used iterated asymptotic cones. In particular, they showed that an asymptotic cone of an asymptotic cone of $X$ is again an asymptotic cone of $X$, and they found $2^{\aleph_0}$ pairwise non-homeomorphic spaces among the iterated asymptotic cones of a suitably chosen group. In other cases, it is easier to show that for a certain space iterated asymptotic cones, rather than all asymptotic cones, are homeomorphic (compare the final remark of \cite{OsS} and \cite[Theorem 0.6]{Si2}).
\par
In this paper we will study iterated asymptotic cones. Examples of interesting behaviors with respect to iterated cones has been found in \cite{Sc} as the result of the fact that all proper metric spaces can be realized as asymptotic cones (see also \cite{Si1}).
\par
Our strategy in this paper is to define a class of metric spaces, the \emph{bullseye spaces}, which are closed under ultralimits and asymptotic cones (with fixed base-point) and which encode 0-1 sequences (well defined up to shift when we consider those spaces up to homeomorphism). Those spaces exhibit a wide range of behaviors with respect to the procedure of iterating the asymptotic cone. It will be easy to have some control on the asymptotic cone of a bullseye space $X$ as the 0-1 sequence associated to the asymptotic cone is an ultralimit of the sequence associated to $X$. The simplest application of this property will be to show that many bullseye spaces have $2^{\aleph_0}$ pairwise non-homeomorphic asymptotic cones (Proposition~\ref{uncountcones}).
\par
Our main theorem is the following ($\Cone_\mu^i(X,e,\alpha)$ denotes the $i-$th iterated asymptotic cone, see Definition~\ref{iteconedef} for details).

\begin{thm}
There exists a group $G$ and a scaling factor $\alpha$, such that for any ultrafilter $\mu$ and any natural numbers $i,j$ with $i \not= j$ the iterated cones $\Cone_\mu^i(G,e,\alpha)$ and $\Cone_\mu^j(G,e,\alpha)$ are not homeomorphic.
\end{thm}

Indeed, we will first construct an example of metric space satisfying the above property (Theorem~\ref{infmany}) and then use techniques from \cite{DS} to obtain an example within the class of groups (Theorem~\ref{mygroup}).
\par
Finally, we will construct other examples within our class displaying interesting behavior, and some of them can be turned into group examples. For example we will show that there are groups with non-trivially periodic iterated asymptotic cones (Theorem~\ref{periodiccones} and Remark~\ref{groupexamples}):

\begin{thm}
 For each positive integer $m$ there exists a group $G$ and a scaling factor $\alpha$, such that for any ultrafilter $\mu$ we have that $\Cone_\mu^i(G,e,\alpha)$ is homeomorphic to $\Cone_\mu^j(G,e,\alpha)$ for $i,j\geq 1$ if and only if $i \equiv j\ \mbox{mod }m$.
\end{thm}

Also, we will provide examples of spaces displaying interesting behaviors with respect to a less rigid definition of iterated cones (Theorem~\ref{asmanyasyouwish}) and transfinite asymptotic cone iteration (Theorem~\ref{transf}), as well as an example of a space with non-homeomorphic iterated cones and ``just'' countably many asymptotic cones (Theorem \ref{onlycount}). 

Large portions of this paper are part of the first author's thesis \cite{Scth} which was written under the supervision of Prof. K. Tent.

\section{Preliminaries}

\subsection{Ultrafilters}

\begin{defi}
Let $I$ be a set. A {\bf filter} $\mu$ on $I$ is a nonempty collection of subsets of $I$, such that for all subsets $A,B \subseteq I$ we have
\begin{itemize}
\item[i)] $\emptyset \notin \mu$.
\item[ii)] $A \in \mu, A \subseteq B \Rightarrow B \in \mu$.
\item[iii)] $A,B \in \mu \Rightarrow A \cap B \in \mu$.
\end{itemize}
The set of all filters on $I$ can be partially ordered by inclusion. It is easy to see that totally ordered subsets have upper bounds and therefore maximal filters exist by Zorn's lemma. Those are called {\bf ultrafilters}. They can be characterized as follows: a filter $\mu$ is an ultrafilter if and only if
\begin{itemize}
\item[iv)] For all $A \subseteq I$ either $A \in \mu$ or $I \setminus A \in \mu$.
\end{itemize}
\end{defi}

An ultrafilter on $I$ can also be regarded as a finitely additive probability measure on $I$, which only takes the values 0 and 1. We say that some property of elements of $I$ holds {\bf $\mu$-almost everywhere} ($\mu$-a.e.) if the set where it holds lies in $\mu$.

\begin{example}
Let $I$ be a set and $i \in I$ a point. Then the collection
\[ \mu_i := \{ A \subseteq I : i \in A \} \]
defines an ultrafilter on $I$. Such an ultrafilter is called {\bf principal}.
\end{example}

Note that for finite sets $I$ each ultrafilter is of this form. Non-principal ultrafilters on $I$ exist if and only if $I$ is infinite: take the collection of all cofinite sets in an infinite $I$. This is a filter and therefore contained in an ultrafilter, which is non-principal since it contains no finite sets.
\par
An ultrafilter $\mu$ on $I$ can be used to assign a limit to any sequence $(x_i)_{i\in I}$ with values in $[0,+\infty]$ (indeed, in any compact Hausdorff topological space). Namely, the $\mu-$limit of $(x_i)$ is the only $a \in [0,+\infty]$ such that every neighborhood of $a$ contains $\mu$-almost every element of the sequence $(x_i)$. Write
\[ \mu-\lim_{i} x_i = a \qquad \mbox{or simply} \quad \mu-\lim x_i = a \]

For later use we also need the definition of product of ultrafilters.

\begin{defi}
Let $I$ be a set and $\mu$ and $\nu$ ultrafilters on $I$. Define the {\bf product} $\mu \times \nu$ on the set $I \times I$ by saying that for $A \subseteq I \times I$ we have
\[ A \in (\mu \times \nu) \iff \{ i \in I : \{j \in I : (i,j) \in A\} \in \nu \} \in \mu.\]
If $I$ is infinite, there is a bijection $ \sigma\colon I \times I \to I$ and we may regard $\mu \times \nu$ again as an ultrafilter on $I$ by taking the preimage of a subset of $I$ under this bijection. Of course the resulting ultrafilter will then depend on the choice of $\sigma$.\\[1ex]
Note that the product is not commutative in general, that is, if $I$ is infinite and $\mu$ and $\nu$ are non-principal ultrafilters on $I$, we might have $\mu \times \nu \not= \nu \times \mu$.
\end{defi}

Suppose that for each pair $i,j \in \IN$, we have a number $x_{ij} \in \IR$ and two ultrafilters $\mu$ and $\nu$ on $\IN$. Then it is easy to see that
\[ \mu-\lim_{i} \big( \nu-\lim_{j} x_{ij} \big) = (\mu\times\nu)-\lim_{(i,j)} x_{ij}.\]

A proof can for example be found in \cite[Lemma 3.22]{DS}.

\subsection{Asymptotic cones}

Let us now define the asymptotic cone of an arbitrary metric space. For more details see \cite{D}.
In what follows if $X$ and $Y$ are metric spaces, the notation $X \cong Y$ will mean that $X$ and $Y$ are isometric.
\begin{defi}
\label{conedefi}
Let $(X,d)$ be a metric space, $\mu$ a non-principal ultrafilter on a countable set $I$, $e=(e_i)$ a sequence of points in $X$ (the \emph{base-point}) and $\alpha=(\alpha_i)$ a sequence of positive real numbers such that $\mu-\lim \alpha_i=+\infty$ (the \emph{scaling factor}\footnote{Sometimes we will first choose a scaling factor and then an ultrafilter. In this case scaling factor will just mean diverging sequence.}).
Consider now the following set:
\[ X^\alpha_e := \left\{ (x_i) \in X^I : \mu-\lim d(x_i,e_i)/\alpha_i<+\infty \right\}.\]
The {\bf asymptotic cone} of $X$ with respect to the base-point $e$, the ultrafilter $\mu$ and the scaling factor $\alpha$ is
$$\Cone_\mu(X,e,\alpha) := X^\alpha_e /_\approx,$$ where
\[ x_n \approx y_n \iff \mu-\lim d(x_n,y_n)/\alpha_i=0.\]
We will denote an equivalence class with respect to $\approx$ by $[x_n]$. The metric $d_\infty$ on $\Cone_\mu(X,e,\alpha)$ is defined by
\[ d_\infty \big([x_n],[y_n]\big) := \mu-\lim d(x_i, y_i)/\alpha_i.\]
\end{defi}

\begin{rmrk}
Sometimes it is convenient to consider more general {\bf $\mu$-limits} of metric spaces. Let $(X_i, d_i)$ be a sequence of metric spaces and consider a point $x = (x_i) \in \prod X_i$. Then the ultralimit of the $X_i$ with basepoint $x$ is defined as the quotient set of
\[ \mu-\lim (X_n, x) := \left\{ (y_i) \in \prod X_i : \mu-\lim d_i(x_i,y_i) < \infty \right\} \]
with respect to the equivalence relation as above. In this light, the construction of the asymptotic cone refers to the special case of setting $(X_i,d_i) := (X, \frac{d}{\alpha_i})$.
\end{rmrk}

\begin{defi}
\label{iteconedef}
[ {\bf Iterated asymptotic cones}.] Fix a non-principal ultrafilter $\mu$ on $\IN$, a scaling factor $\alpha$. For each metric space $X$ and $e\in X$, set $\Cone^0_\mu(X,e,\alpha) := X$, $e(0)=e$ and for $i \in \IN$ set
\[ \Cone^{i+1}_\mu(X,e,\alpha) := \Cone_\mu\big(\Cone^i_\mu(X,e,\alpha),e(i),\alpha \big),\]
\[ e(i+1)=[\widehat{e(i)}], \]
where $\widehat{e(i)}$ is the constant sequence with value $e(i)$.
\par
In order to simplify the notation we will denote each $e(i)$ for $i\geq 1$ by $\hat{e}$.
 
\end{defi}

The following lemma, which can be found in \cite{DS}, Section 3.2, is crucial.
\begin{lemma}
\label{coneofcone}
Let $(X,d)$ be a metric space, $e \in X$ a basepoint and fix two non-principal ultrafilters $\mu$ and $\nu$ on $\IN$ and scaling factors $\alpha$ and $\beta$. Then
\[ \Cone_\mu\big(\Cone_\nu(X,e,\alpha),\hat{e},\beta\big) \cong \Cone_{\mu \times \nu}(X,e,\gamma),\]
where $\gamma$ is the sequence of real numbers indexed by $\IN \times \IN$ defined as
\[ \gamma_{k,n} := \alpha_n \beta_k.\]
\end{lemma}

In particular, if the asymptotic cones of $X$ are all isometric, then so are the iterated asymptotic cones.

\section{Infinite iteration}

We want to give an example of a metric space $X$ having infinitely many pairwise non-homeomorphic iterated cones, which means that for every $i \not= j$ the space $\Cone^i_\mu(X,e,\alpha)$ is not homeomorphic to $\Cone^j_\mu(X,e,\alpha)$.\\
In order to do so we define a family of metric spaces which encode 0-1 sequences in a suitable way. A similar idea has been exploited by Bowditch to show that there are uncountably many quasi-isometry classes of groups \cite{B}.

\begin{defi}
\label{bullseyedef}
Fix a sequence $(a_k)_{k \in \IZ} \in \{0,1\}^\IZ$. We will encode this sequence in a metric space $X$, called {\bf bullseye space} associated to the sequence $(a_k)$. Consider the union of all circles in $\IR^2$ with radii $2^k$, $k \in \IZ$ all centered at the origin and add the origin to the space as a basepoint, called $e$.\\
Now, add the $x-$axis, which can be sees as the union of two rays from the origin to infinity. On one of the rays we put discs on every interval between two circles, rescaled in such a way that the space stays scaling invariant for powers of 2. On the other ray we do the same for 3-dimensional balls.\\
For each $k \in \IZ$, connect the circle of radius $2^k$ to the circle of radius $2^{k+1}$ with a suitable segment contained in the positive part of the $y-$axis if and only if $a_k = 1$. These segments are called {\bf bridges}.\\
Finally, endow the resulting space $X$ with the path metric.

\end{defi}

\bullseyefigure

\begin{rmrk}
 $X$ is geodesic, proper and it does not contain global cut-points (see Definition~\ref{cutp}).
\end{rmrk}

\begin{lemma}
\label{bullseyecone}
Fix a sequence $(a_k)$ as above and a set \[ A = \{ \alpha_0 < \alpha_1 < \cdots \} \subseteq \{2^n : n \in \IN\}.\] 
Set $\alpha := [\alpha_n]$ and fix any ultrafilter $\mu$. Denote the bullseye space associated to $(a_k)$ by $X$. Then $\Cone_\mu(X,e,\alpha)$ will be isometric to a bullseye space associated to the sequence $(b_k)$ given by
\[ b_k = \lim_\mu a_{\alpha_n + k}.\]
\end{lemma}

\begin{proof}
When rescaled by a power of 2, $X$ will still contain circles of length $2^k\pi$ for each $k$, a ``central'' point $e$ and the ``decorated'' rays we described above. Therefore, the same holds true for each asymptotic cone of $X$ (with basepoint $e$). Also, such asymptotic cone has a bridge between the circle of $2^k$ and $2^{k+1}$ if and only if the set of rescaled spaces having a bridge between $2^{\alpha_n + k}$ and $2^{\alpha_n + k + 1}$ has measure 1 with respect to $\mu$. This proves the assertion.
\end{proof}

We want to be able to distinguish the spaces corresponding to sequences we use and for this purpose we will use the following invariant.

\begin{defi}
Let $(a_k)$ be a sequence as above. The {\bf asymptotic density} of $(a_k)$ is defined as
\[ \adn(a_k) :=  \limsup_{n \to \infty} \frac{1}{2n + 1} \cdot \sum_{k=-n}^n a_k.\]
\end{defi}

\begin{lemma}
\label{asdense}
Let $(a_k)$ be a 0-1 sequence. Fix $N \in \IN$ and consider the shifted sequence $b_k := a_{k+N}$. Then
\[ \adn(b_k) = \adn(a_k).\]
\end{lemma}

\begin{proof}
For $n > N$ we have
\[ \left| \frac{1}{2n+1}\left( \sum_{k=-n}^n a_k - \sum_{k=-n+N}^{n+N} a_k \right) \right| = \left| \frac{1}{2n+1}\left( \sum_{k=-n}^{-n+N}a_k - \sum_{k=n}^{n+N} a_k \right) \right| \leq \frac{N}{2n+1}\]
and this tends to 0 for fixed $N$ and $n \to \infty$, therefore $\adn(a_k) = \adn(a_{k+N})$.
\end{proof}

\begin{defi}
\label{thinsequence}
A set $A \subseteq \IN$ given by $A = \{ \alpha_0 < \alpha_1 < \alpha_2 < \cdots \}$ is called {\bf thin} if
\[ \lim_{n \to \infty} \frac{\alpha_{n+1}}{\alpha_n} = \infty.\]
\end{defi}

For example, the set $\{ n! : n \in \IN \}$ is thin.
We can now state and prove the main theorem.

\begin{thm}
\label{infmany}
There exists a metric space $X$ with basepoint $e$ and a scaling factor $\alpha$, such that for any ultrafilter $\mu$ and any natural numbers $i,j$ with $i \not= j$ the iterated cones $\Cone_\mu^i(X,e,\alpha)$ and $\Cone_\mu^j(X,e,\alpha)$ are not homeomorphic.
\end{thm}

\begin{proof}
Fix any ultrafilter $\mu$ on $\IN$. Take a thin set $A = \{\alpha_0 < \alpha_1 < \cdots \} \subseteq \{2^n : n \in \IN\}$ and $\alpha := [\alpha_n]$. For any sequence $(a_k) \in \{0,1\}^\IZ$, call the numbers $a_k$ with $k \in [ \alpha_n - n, \alpha_n + n]$ for some $n \in \IN$ the {\bf variable part} of the sequence. Its complement will be called the {\bf fixed part}. Let us also assume $\alpha_0 \gg 0$ and that the intervals given above are disjoint.\\
For every $i \in \IN$ define now a sequence $(a_k^{(i)})$ with $\adn(a_k^{(i)}) = 1/(i+1)$ (the only property we actually need is that these densities are different). Note that since the set $A$ is thin, the density will still be defined and will have the same value if you modify the sequence $(a_k^{(i)})$ on the variable part, since the relative amount of the variable part in any given interval of the form $[-n,n]$ in the sequence tends to 0 as $n$ goes to infinity.\\
Next, modify the sequence $(a_k^{(0)})$ in such a way that the cone of the bullseye space $X$ associated to $(a_k^{(0)})$ is the bullseye space associated to $(a_k^{(1)})$. By Lemma \ref{bullseyecone}, it is enough to modify $(a_k^{(0)})$ on the variable part, not changing its density.\\
Then iterate this process, modifying the variable part of $(a_k^{(i)})$ in such a way that the cone of the bullseye space associated to this sequence is the bullseye space associated to $(a_k^{(i+1)})$. This change has to be reflected in all the $(a_k^{(j)})$ with $j < i$ as well. Since by assumption $\alpha_0 \gg 0$, this process yields a well-defined limit sequence, which we denote by $(a_k^{[i]})$, because every fixed entry in any given sequence is modified only finitely many times.\\
Now, define the space $X$ as the bullseye space associated to the sequence $(a_k^{[0]})$. It is easy to see that if two bullseye spaces are homeomorphic, they correspond to the same underlying sequence, up to a shift, because the rays can only be sent to the same rays using a homeomorphism. From this, it follows that for any numbers $i,j \in \IN$ with $i \not= j$, the spaces $\Cone_\mu^i(X,e,\alpha)$ and $\Cone_\mu^j(X,e,\alpha)$ cannot be homeomorphic by Lemma \ref{asdense}, since the underlying sequences have different asymptotic densities.
\end{proof}

\section{Iterated cones of groups}

\subsection{Tree-graded spaces}

To give the desired example of a finitely generated group with countably many pairwise non homeomorphic iterated asymptotic cones, we use a result by Dru\c{t}u and Sapir from \cite{DS}. To state and explain this result, we first need the notion of tree-graded spaces.

\begin{defi}[\cite{DS}, Definition 2.1]
Let $X$ be a complete geodesic metric space and let $\mathcal{P}$ be a collection of closed geodesic subsets, called {\bf pieces}, which cover the space $X$. We say that $X$ is {\bf tree-graded} with respect to $\mathcal{P}$ if
\begin{itemize}
\item[(T1)] The intersection of any two different pieces is either empty or a single point.
\item[(T2)] Every simple geodesic triangle in $X$ is contained in one piece.
\end{itemize}
\end{defi}

The second property can be substituted by ``every simple loop in $X$ is contained in one piece'', providing a topological characterization of tree-graded spaces.
\par
We will need another result of Dru\c{t}u and Sapir, stating that there exists for each tree-graded space a minimal set of pieces.

\begin{defi}
Let $X$ be a metric space which is tree-graded with respect to two sets of pieces $\mathcal{P}$ and $\mathcal{P}'$. Write $\mathcal{P} \prec \mathcal{P}'$ if for every $A \in \mathcal{P}$ there is a piece $A' \in \mathcal{P}'$ such that $A \subseteq A'$. Note that this defines a partial order.
\end{defi}

\begin{defi}
\label{cutp}
Let $X$ be a geodesic metric space. A point $x \in X$ is called a {\bf global cut-point} of $X$ if the space $X \backslash \{x\}$ is not path connected.
\end{defi}

\begin{lemma}[\cite{DS}, Lemma 2.31]
\label{cutpoints}
Let $X$ be a complete geodesic space containing at least two points.
There exists a unique minimal tree-graded structure $\mathcal{P}$ for $X$ (with respect to $\prec$), such that $X$ is tree-graded with respect to $\mathcal{P}$ and any piece in $\mathcal{P}$ is either a singleton or a set $P$ with no global cut-point.
\end{lemma}

\begin{defi}[cf. \cite{DS}, Definition 3.19]
Let $X$ be a metric space with basepoint $e \in X$. Fix an ultrafilter $\mu$ on $\IN$ and a scaling sequence $\alpha$. Let $\mathcal{A}$ be a collection of subsets of $X$. Then for every sequence $(A_n)$ of sets in $\mathcal{A}$, the set
\[ \Cone_\mu\big((A_n),e,\alpha\big) := \{[x_n] \in \Cone_\mu(X,e,\alpha) : x_n \in A_n \}\]
is a (possibly empty) subset of the asymptotic cone of $X$. We say that $X$ is {\bf asymptotically tree-graded} with respect to $\mathcal{A}$ if $\Cone_\mu(X,e,\alpha)$ is tree-graded with respect to the set of non-empty pieces of the form
\[ \{ \Cone_\mu\big((A_n),e,\alpha\big) : (A_n)_{n \in \IN} \in \mathcal{A}^\IN \}.\]
\end{defi}

We will also need the following useful fact, see \cite[Theorem 3.30]{DS}.

\begin{lemma}
\label{ittreegraded}
Let $X$ be a geodesic metric space, which is tree-graded with respect to a collection of pieces $\mathcal{P}$. Then $X$ is also asymptotically tree-graded with respect to the same set of pieces.
\end{lemma}

\subsection{Infinite iteration for groups}

We will use the following result, which is proven (though not stated) in \cite[Section 7]{DS}.

\begin{defi}
 A \emph{tame exhaustion} of a geodesic metric space $X$ is a family $\{B_n\}_{n\in\IN}$ of subsets of $X$ with the property that for each each ball $B$ in $X$ there exists $N$ such that, for each $n\geq N$, $B$ is contained in $B_n$. Also, we require that for each $x,y\in B_n$ there exists a geodesic connecting them which is contained in $B_{2n}$.
\end{defi}

\begin{thm}[\cite{DS}, Proposition 7.26, Proposition 7.27, Lemma 7.5]
\label{drutugroup}
Let $X$ be a proper geodesic metric space and $e\in X$. There exists a scaling factor $\alpha=(\alpha_n)$ and a group $G$ with 2 generators such that for every ultrafilter $\mu$ the asymptotic cone $\Cone_\mu(G,e,\alpha)$ is tree-graded with respect to pieces whose collection of isometry classes coincides with
\[  \left\{ \mu-\lim (B_n, x) : x \in \prod B_n \right\}, \]
where $\{B_n\}$ is a tame exhaustion of $X$ and each $B_n$ contains $e$. Also, $\mu-\lim (B_n, e)\cong X$. 
\end{thm}

A quick word on the proof: Proposition 7.26 and Proposition 7.27 from \cite{DS} give an analogous result where $X$ is substituted by a sequence of graphs, and $\alpha$ is chosen to be ``fast increasing'' with respect to those graphs. In order to conclude it is enough to show that $X$ is well approximated by a certain sequence of graphs (see the discussion in \cite{DS} leading to Lemma 7.5). This last step can be carried out considering certain configurations of points in each $B_n$ called ``$\delta$-snets''. In \cite{DS} $B_n$ is taken to be the ball of radius $n$ around $e$, but the proof goes through unchanged if $\{B_n\}$ is a tame exhaustion (the last condition in the definition of tame exhaustion is used in \cite[Lemma 7.5$-(1)$]{DS}).

\begin{rmrk}
\label{exhaust}
We will apply the theorem letting $X$ be a bullseye space. Also, we will always choose as $B_n$ a ``truncated'' bullseye space, that is a space constructed as in Definition \ref{bullseyedef}, except that we start with a sequence $(a_k)_{k \in \IZ, k\leq n}$ (and we consider circles of radii $2^k$ for $k\leq n+1$ and a suitable segment instead of the full $x-$axis). It is more convenient to choose this tame exhaustion instead of using balls of integer radius especially for the purposes of Remark \ref{groupexamples}.
\end{rmrk}
We will also need the following remark, which follows from the proof of the theorem (the first part of it being just a technicality).
\begin{rmrk}
 $\alpha$ can be chosen to be a subset of $\{2^n\}_{n\in\IN}$. Moreover, we can choose the same sequence $\alpha$ for all bullseye spaces.
\end{rmrk}

The second part holds because for each bullseye space, each $n$ and each $\delta$ there is a uniform bound on the cardinality of a $\delta$-snet in $B_n\subseteq X$, and therefore a bound on the number of edges in the approximating graphs. Also, the properties required for a sequence to be fast increasing only depend on the number of edges of the graphs.
\par
Now, in order to use Theorem \ref{drutugroup} for our purpose, we consider a bullseye space $X$ like the one constructed in Theorem \ref{infmany} having infinitely many non-homeomorphic iterated cones. Fix a sequence $(\alpha_n)$ as above. 

\begin{thm}
\label{mygroup}
There exists a finitely generated group $G$ and a scaling factor $\alpha$, such that for every ultrafilter $\mu$ and every two numbers $i,j \in \IN$ with $i \not= j$ we have that $\Cone^i_\mu(G,e,\alpha)$ is not homeomorphic to $\Cone^j_\mu(G,e,\alpha)$.
\end{thm}

\begin{proof}
Take $X$ as above from Theorem \ref{infmany}, using the sequence $\alpha$ to define variable and fixed parts. Let $G$ be a group as in Theorem \ref{drutugroup}. Fix an ultrafilter $\mu$ and set $C := \Cone_\mu(G,e,\alpha)$.\\
Consider on $C$ the tree-graded structure $\mathcal{P}$ given by Lemma \ref{cutpoints}. As $X$ does not contain cut-points and it is a piece in the tree-graded structure given by Theorem \ref{drutugroup}, it appears as a piece in $\mathcal{P}$. However, in $C$ we also have other pieces coming from ultralimits of $B_n$ with varying sequences of basepoints. Observe that for every possible base-point $(x_n) \in \prod B_n$ we are in one of two cases: either $\mu-\lim d(x_n,e)/\alpha_n<+\infty$ (in which case the cone with basepoint $x$ is $X$) or $\mu-\lim d(x_n, e)/\alpha_n=+\infty$. But in the second case the cone with respect to this $x$ can only contain 2-dimensional or 3-dimensional parts, but not both. In particular such a piece cannot be homeomorphic to any bullseye space, and nor can any rescaled ultralimit of those.\\
In particular, we have that $X$ can be topologically characterized as the only bullseye space appearing as a piece in the minimal tree-graded structure on $C$.\\
Now we iterate the process, taking the asymptotic cone of $C$, say $C'$. By Lemma \ref{ittreegraded}, $C'$ will again be tree-graded and the pieces will be rescaled ultralimits of pieces. In particular the cone of $X$ will occur as the only bullseye piece. Since all iterated cones of $X$ are pairwise non-homeomorphic, in particular the iterated cones of $G$ will be pairwise non-homeomorphic as well.
\end{proof}

\section{Some variations}
\label{variations}
In this section we show how to modify some of the constructions we performed in order to obtain some further examples. Many proofs are obtained by suitably modifying the proof of Theorem~\ref{infmany} and use the same notation.

\subsection{A space with uncountably many cones}

In this subsection we will show that ``most'' bullseye spaces have uncountably many pairwise non-homeomorphic asymptotic cones.

\begin{defi}
Let $(a_k)_{k \in \IZ}$ be a sequence in $\{0,1\}^\IZ$. This sequence is called {\bf rich} if it contains every finite sequence of the numbers 0 and 1 in its positive part $(a_k)_{k \in \IN}$.
\end{defi}

Clearly, rich sequences exist since there are only countably many finite sequences. Also note that a rich sequence will contain every given finite sequence infinitely many times, since any given finite sequence can be extended in infinitely many ways to different longer finite sequences.

\begin{prop}
\label{uncountcones}
If $X$ is a bullseye space such that its associated sequence $(a_k)$ is rich then for each ultrafilter $\mu$ there exists a set $\mathcal{S}$ of scaling factors with $|\mathcal{S}| = 2^{\aleph_0}$ such that for every $\alpha, \alpha' \in \mathcal{S}$ with $\alpha \not= \alpha'$ the spaces $\Cone_\mu(X,e,\alpha)$ and $\Cone_{\mu}(X,e,\alpha')$ are not homeomorphic.
\end{prop}

\begin{proof}
For every $t \in [0,1]$ choose a sequence $(a_k^{(t)})_{k \in \IZ}$ with $\adn(a_k^{(t)}) = t$.\\
Fix $t \in [0,1]$. We can construct a set of indices $i_1^{(t)} <  i_2^{(t)} < i_3^{(t)} <  \ldots$ in $\IN$, such that for all $n \in \IN$ and $l \in \IZ$ with $-n \leq l \leq n$ we have
\[ a_{i_n^{(t)} + l} = a_l^{(t)}.\]
This is clearly possible since $(a_k)$ is rich and these finite sequences all occur infinitely often. Let $\alpha_t$ be be the sequence $i_1^{(t)}, i_2^{(t)}, \ldots$. By construction, it follows that the space $\Cone_{\mu}(X,e,\alpha_t)$ is again a bullseye space with associated sequence $(a_k^{(t)})$. Set $\mathcal{S} := \{ \alpha_t : t \in [0,1] \}$. Since all the resulting sequences have different asymptotic densities, the claim follows.
\end{proof}

\begin{rmrk}
In \cite[Theorem 1.10]{KSTT}, the authors showed that $2^{\aleph_0}$ is the maximal number of asymptotic cones a finitely generated group can have, provided the continuum hypothesis $(CH)$ is true. However, their proof does not use the group structure at all and works exactly the same way for arbitrary metric spaces of cardinality at most $2^{\aleph_0}$. So, even in the more general context of arbitrary metric spaces there can only be $2^{\aleph_0}$ different cones, provided $(CH)$ holds.
\end{rmrk}

\subsection{Periodic iterated cones}

\begin{thm}
\label{periodiccones}
For any positive integer $m$ there exists a metric space $X$ with basepoint $e$ and a scaling factor $\alpha$, such that for any ultrafilter $\mu$ and any natural numbers $i,j$ the iterated cones $\Cone_\mu^i(X,e,\alpha)$ and $\Cone_\mu^j(X,e,\alpha)$ are homeomorphic if and only if $i \equiv j\ \mbox{mod }m$.
\end{thm}

\begin{proof}
 We can proceed as in the proof of Theorem~\ref{infmany}, starting off with sequences $(a^{(i)}_k)$ such that
\begin{itemize}
 \item $a^{(i)}_k=a^{(i+m)}_k$,
 \item if $i \not\equiv j\ \mbox{mod }m$.
\end{itemize}

It is easily seen that the sequences $a^{[i]}_k$ satisfy the same properties.

\end{proof}

\begin{rmrk}
\label{groupexamples}
Using the same techniques as in the previous section it is possible, starting from the example above, to construct a group $G$ with non-trivially periodic iterated cones, meaning that $\Cone_\mu^i(G,e,\alpha)$ and $\Cone_\mu^j(G,e,\alpha)$, for $i,j\geq 1$, are homeomorphic if and only if $i \equiv j\ \mbox{mod }m$. The extra ingredient we need is \cite[Theorem 0.6]{Si2}, which in our context gives that $\Cone_\mu^i(G,e,\alpha)$ is (bilipschitz) homeomorphic to $\Cone_\mu^j(G,e,\alpha)$ ($i,j\geq 1$) if and only if their pieces in the minimal tree-graded structure are bilipschitz homeomorphic. It is easily seen that, when $G$ is constructed as in the previous section using the bullseye space $X$ as in the theorem above, each such piece is bilipschitz homeomorphic to either $\IR^2$, $\IR^3$, a half-plane in $\IR^2$, a half-space in $\IR^3$ or an iterated cone of $X$. In particular, $\Cone_\mu^i(G,e,\alpha)$ and $\Cone_\mu^j(G,e,\alpha)$ ($i,j\geq 1$) are homeomorphic if and only if $\Cone_\mu^i(X,e,\alpha)$ and $\Cone_\mu^j(X,e,\alpha)$ are homeomorphic.
\end{rmrk}

\begin{rmrk}
With a similar method it is possible to construct spaces $X$ whose first, say, $k$ iterated asymptotic cones are pairwise non-homeomorphic, and the following ones are periodic.
\end{rmrk}

\subsection{Changing scaling factor/ultrafilter}

Recall that in the definition of iterated cone we fixed an ultrafilter and a scaling factor. Allowing of them to vary gives new possible behaviors. Indeed, we obtain the maximal possible range of behaviors by just allowing one of them to vary.

\begin{defi}
Fix a sequence of non-principal ultrafilters $(\mu_n)_{n\in\IN}$ on $\IN$ and a sequence of scaling factors $(\alpha_n)_{n\in\IN}$.
We set $\Cone^0_{(\mu_n)}(X,e,(\alpha_n)) := X$ and for $i \in \IN$ set
\[ \Cone^{i+1}_{(\mu_n)}(X,e,(\alpha_n)) := \Cone_{\mu_{i+1}}\big(\Cone^i_{\mu_i}(X,e,(\alpha_n)),\hat{e},\alpha_{i+1} \big).\] 
\end{defi}

We recover the definition of ``regular'' iterated cones by setting $\mu_n=\mu$ and $\alpha_n=\alpha$ for each $n$. We will denote $\Cone^{i}_{(\mu_n)}(X,e,(\alpha_n))$ by $\Cone^{i}_{\mu}(X,e,(\alpha_n))$ if $(\mu_n)$ is the sequence with constant value $\mu$.

\begin{thm}
\label{asmanyasyouwish}
 Let $(a_n)_{n\in\IN}$ be any sequence of natural numbers.
There exists a metric space $X$ with basepoint $e$ and a sequence of scaling factors $(\alpha_n)$, such that for any ultrafilter $\mu$ we have that the iterated cones $\Cone_\mu^i(X,e,(\alpha_n))$ and $\Cone_\mu^j(X,e,(\alpha_n))$ are homeomorphic if and only if $a_i=a_j$.
\end{thm}

\begin{proof}
 Choose sequences $\{\beta_i=(b^{(i)}_n)\}_{i\in\IN}$ with values in $\{0,1\}$ with the following properties:
\begin{enumerate}
 \item each sequence $\beta_i$ is rich,
 \item if $a_i\neq a_j$ then $\beta_i$ and $\beta_j$ have different asymptotic densities,
 \item if $a_i=a_j$ then $\beta_i=\beta_j$.
\end{enumerate}

 We can choose $X$ to be the bullseye whose associated sequence is $\beta_0$. It is then possible, given $\mu$, to choose inductively scaling factors $\alpha_i$ so that $\Cone^{i}_{\mu}(X,e,(\alpha_n))$ is the bullseye associated to the sequence $\beta_i$.
\end{proof}

\subsection{Transfinite iteration}
There is a natural definition of $\Cone_\mu^\lambda(X,e,\alpha)$, where $\lambda$ is an ordinal of the form $\omega\cdot k+n$, $k,n\in<\IN$, namely we set
$$\Cone_\mu^{\omega\cdot (k+1)}(X,e,\alpha)=\mu-\lim \Cone_\mu^{\omega\cdot k+n}(X,e,\alpha).$$

\begin{thm}
\label{transf}
For each positive integer $k$ there exists a metric space $X$ with basepoint $e$ and a scaling factor $\alpha$, such that for any ultrafilter $\mu$ and distinct ordinals $\lambda_1,\lambda_2<\omega\cdot k$ the iterated cones $\Cone_\mu^{\lambda_1}(X,e,\alpha)$ and $\Cone_\mu^{\lambda_2}(X,e,\alpha)$ are not homeomorphic.
\end{thm}

\begin{proof}
 As usual, $X$ will be a bullseye space. In order to set the case $k=1$, notice that we can arrange that $\Cone_\mu^{\omega}(X,e,\alpha)$, where $X$ is the the bullseye we constructed in the proof of Theorem~\ref{infmany}, to be any bullseye we want. In fact, it is enough to require the sequences $(a_k^{(i)})$ we started from to have larger and larger subsequences centered in $0$ to coincide with a given sequence, and not modify that part of the sequence (this does not affect asymptotic cones of the associated bullseye spaces).
\par
This not just settles the case $k=1$, but also provides the inductive step needed: given a bullseye $X$ with required property for a certain $k$, we can find another bullseye space $X'$ such that $\Cone_\mu^{\omega}(X',e,\alpha)$ is $X$, and we can also arrange that $\Cone_\mu^{i}(X',e,\alpha)$ is not homeomorphic to $\Cone_\mu^{\lambda}(X,e,\alpha)$ for $i$ finite and $\lambda<\omega \cdot k$.

\end{proof}

\subsection{Only countably many asymptotic cones}

In this subsection we will provide an example of a bullseye space which, despite having non-homeomorphic asymptotic cones, has ``only'' countably many distinct asymptotic cones.

\begin{thm}
\label{onlycount}
There exists a metric space $X$ with basepoint $e$ and a scaling factor $\alpha$ such that for any ultrafilter $\mu$ the following hold:
\begin{itemize}
 \item for any natural numbers $i,j$ with $i \not= j$ the iterated cones $\Cone_\mu^i(X,e,\alpha)$ and $\Cone_\mu^j(X,e,\alpha)$ are not homeomorphic,
 \item $X$ has countably many asymptotic cones up to homeomorphism.
\end{itemize}

\end{thm}

\begin{proof}
 Notice that there are finitely many homeomorphism classes of asymptotic cones of bullseye spaces with base-point $(x_n)$ and scaling factor $\alpha_n$ such that $\mu-\lim d(x_n,e)/\alpha_n=+\infty$. So, we only need to consider asymptotic cones with basepoint $e$. In particular, consider for each natural number $i$ the sequence $(a^{(i)}_k)$ such that $a^{(i)}_k=1$ if $k$ is divisible by $i+2$ and $a^{(i)}_k=0$ otherwise. Those sequences clearly have pairwise different asymptotic density. Also, all the asymptotic cones of a bullseye space $Y$ associated to any of those sequence (with basepoint $e$) is homeomorphic to $Y$.
\par
Up to considering a ``sparser'' set than $A$, we can replace the intervals $[\alpha_n-n,\alpha_n+n]$ with intervals $[\alpha_n-k(n),\alpha_n+k(n)]$ for a suitable sequence $k(n)\geq n$ in such a way that for each $n,m$ we have:
 $$\alpha_m-k(m)\leq k(n)\Rightarrow \alpha_{m}+k(m)\leq k(n)-n,$$
 $$\alpha_m+k(m)\geq k(n)\Rightarrow \alpha_{m}-k(m)\geq k(n)+n. $$
In fact, we can choose inductively $\alpha_j, k(j)$ with $k_j\ll \alpha_j$ but $k_j,\alpha_j-k_j\gg \alpha_{j-1}+k(j-1)$.
The condition gives that if $I_{n,m}=\alpha_n+[\alpha_m-k(m),\alpha_m+k(m)]$ intersects $[\alpha_n-k(n),\alpha_n+k(n)]$, then $I_{n,m}\subseteq [\alpha_n-k(n)+n,\alpha_n+k(n)-n]$.
Roughly speaking, this ensures that we first modify $(a^{(0)}_k)$ in certain intervals, then in certain other intervals ``well-inside'' the given ones and so on.
\par
Consider now the bullseye space $X$ with associated sequence $(a^{[0]}_k)$. We claim that each asymptotic cone of $X$ (with basepoint $e$) is a bullseye whose associated sequence is either
\begin{enumerate}
 \item a sequence containing at most two $1$, or
 \item a sequence obtained concatenating an initial subsequence of $(\alpha^{(i)}_k)$ and a (possibly empty) final subsequence of $(\alpha^{(i+1)}_k)$, for some $i$.
\end{enumerate}

This clearly implies that $X$ is as required.
\par
In order to show this associate to each $n$ the maximal number $t(n)$ such that there exist $m_0,\dots,m_{t(n)}$ such that $|n-\alpha_{m_0}|\leq k(m_0), |n-\alpha_{m_0}-\alpha_{m_1}|\leq k(m_1),\dots$ (in a not precise but more evocative way: the number of variable parts that $n$ belongs to). Notice that the finite sequence $(a^{[0]}_{k})|_{J_n}$, where $J_n$ is the interval around $n$ of radius $t(n)$ can be obtained concatenating restrictions of $(a^{(t(n)-1)}_{k})$ and $(a^{(t(n))}_{k})$ to suitable finite intervals.
\par
Let $\alpha=(\alpha_n)$ be any scaling factor. There are two cases to consider.
\par
Suppose that $\mu-\lim t(\lfloor\log_2(\alpha_n)\rfloor)=+\infty$. Notice that $(a^{[0]}_{k})$ takes value $0$ at most twice in an interval of radius $t(n)$ around $n$. In particular, the asymptotic cone of $X$ with scaling factor $\alpha$ will be as in case $(1)$ above.
\par
Suppose now that $\mu-\lim t(\lfloor\log_2(\alpha_n)\rfloor)=N<+\infty$. In this case the description of $(a^{[0]}_{k})|_{J_n}$ given above guarantees that the asymptotic cone of $X$ with scaling factor $\alpha$ is as in case $(2)$.

\end{proof}


\begin{thebibliography}{}
\bibitem[B]{B}B. H. Bowditch, {\it Continuously many quasiisometry classes of 2-generator groups.} Comment. Math. Helv. {\bf 73} (1998), 232--236.
\bibitem[D]{D} C. Dru\c{t}u, {\it Quasi-isometry invariants and asymptotic cones.} Int. J. Alg. Comp. {\bf 12} (2002), 99--135.
\bibitem[dC]{dC}Y. de Cornulier {\it Dimension of asymptotic cones of Lie groups.} J. Topol. {\bf 1} (2008), no. 2, 342–361.
\bibitem[DP]{DP}A. Dyubina, I. Polterovich, {\it Explicit constructions of universal $\IR$-trees and asymptotic geometry of hyperbolic spaces.} Bull. London Math. Soc. {\bf 33} (2001), 727--734.
\bibitem[DS]{DS}C. Dru\c{t}u, M. Sapir, {\it Tree-graded spaces and asymptotic cones of groups.} Topology {\bf 44} (2005), 959--1058.
\bibitem[G]{G1}M. Gromov, {\it Groups of polynomial growth and expanding maps.} Publications Math{\'e}matiques de l'IH{\'E}S {\bf 53} (1981), 53--78.
\bibitem[KSTT]{KSTT}L. Kramer, S. Shelah, K. Tent, S. Thomas, {\it Asymptotic cones of finitely presented groups.} Advances in Mathematics {\bf 193} (2005), 142--173.
\bibitem[OlS]{OS}A. Y. Ol'shanskii, M. V. Sapir, {\it A finitely presented group with two non-homeomorphic asymptotic cones.} Int. J. Alg. Comp. {\bf 17} (2007), no. 2, 421--426.
\bibitem[OsS]{OsS}D. V.Osin, M. V. Osin, {\it Universal tree-graded spaces.} To appear in Int. J. Alg. Comp.
\bibitem[P]{P} P. Pansu, {\it Croissance des boules et des g\'eod\'esiques ferm\'ees dans les nilvari\'et\'es.} Ergodic Theory Dynam. Systems {\bf 3} (1983), no. 3, 415–445.
\bibitem[Sc1]{Sc}L. Scheele, {\it Slow ultrafilters and asymptotic cones of proper metric spaces}\\ arXiv:1010.1699v1 (2010).
\bibitem[Sc2]{Scth}L. Scheele, {\it Iterated asymptotic cones.} Phd Thesis, Universit\"at M\"unster (2011).
\bibitem[Si1]{Si1}A. Sisto, {\it Separable and tree-like asymptotic cones of groups.}\\ arXiv:1010.1199v2 (2010).
\bibitem[Si2]{Si2}A. Sisto, {\it Projections and relative hyperbolicity.}\\ arXiv:1010.4552v3 (2011).
\bibitem[TV]{TV}S. Thomas, B. Velickovic, {\it Asymptotic cones of finitely generated groups.} Bull. London Math. Soc. {\bf 32} (2000), 203--208.
\bibitem[vDW]{VDW}L. van den Dries, A. J. Wilkie, {\it On Gromov's theorem concerning groups of polynomial growth and elementary logic.} J. of Algebra {\bf 89} (1984), 349--374.
\end{thebibliography}
\end{document}